\documentclass[12pt]{article}
\usepackage{}

\usepackage{epsfig}
\usepackage{latexsym}
\usepackage{caption}
\usepackage{amsfonts}
\usepackage{amssymb}
\usepackage{mathrsfs}
\usepackage{amsmath}
\usepackage{enumerate}
\usepackage{graphics}
\usepackage{MnSymbol}
\usepackage{float}
\usepackage{pict2e}
\usepackage{subfigure}

\makeatletter

\newcommand{\Rmnum}[1]{\expandafter\@slowromancap\romannumeral #1@}

\makeatother

	\newtheorem{theorem}{Theorem}
	\newtheorem{observation}[theorem]{Observation}

	\newtheorem{definition}{Definition}
	
	\newtheorem{claim}[theorem]{Claim}
	
	\newtheorem{question}[theorem]{Question}
	\newtheorem{lemma}[theorem]{Lemma}

	\makeatother

	\newenvironment{proof}{\noindent {\bf
			Proof.}}{\rule{3mm}{3mm}\par\medskip}

	\newcommand{\q}{\uppercase\expandafter{\romannumeral1}}
	\newcommand{\qq}{\uppercase\expandafter{\romannumeral2}}
	\newcommand{\qqq}{\uppercase\expandafter{\romannumeral3}}
	\newcommand{\qqqq}{\uppercase\expandafter{\romannumeral4}}
	\newcommand{\qqqqq}{\uppercase\expandafter{\romannumeral5}}
	\newcommand{\qqqqqq}{\uppercase\expandafter{\romannumeral6}}

\begin{document}
		
\title{On Hedetniemi's conjecture and the Poljak-R\"odl function\footnote{This note is based on an invited lecture at the MATRIX program
		``Structural Graph Theory Downunder'' held in Creswick, Australia,
		Nov. 24--Dec. 1st, 2019. The main part of the lecture is to explain Shitov's proof, which is based on Shitov's paper, but  more explanations are added. }}
\author{ 
	Xuding Zhu\thanks{Zhejiang Normal University, 	
		email:xdzhu@zjnu.edu.cn. 
		Grant Numbers: NSFC 11971438 and 111 project of Ministry of Education of China.}}


	\maketitle
	\begin{abstract}
		
		Hedetniemi conjectured in 1966 that $\chi(G \times H) = \min\{\chi(G), \chi(H)\}$
		for any graphs G and H.  Here $G\times H$ is the graph with vertex set
	$	V(G)\times V(H)$ defined by putting $(x,y)$ and $(x',y')$ adjacent if and only if
		$xx'\in E(G)$ and $yy'\in V(H)$. This conjecture received a lot of attention in the past half century. 
		It was disproved recently by Shitov. 
		The Poljak-R\"{o}dl function is defined as $f(n) = \min\{\chi(G \times H): \chi(G)=\chi(H)=n\}$. 
		Hedetniemi's conjecture is equivalent to saying $f(n)=n$  for all integer $n$. Shitov's result shows that 
		$f(n)<n$ when $n$ is sufficiently large.  Using Shitov's result, Tardif and Zhu showed that  $f(n) \le n - (\log n)^{1/4}$ for sufficiently large  $n$. Using Shitov's method, He--Wigderson showed that for $\epsilon \approx 10^{-9}$ and $n$ sufficiently large, $f(n) \le (1-\epsilon)n$. In this note we show that a slight modification of the proof in the paper of Zhu and Tardif   shows that  $f(n) \le (\frac 12 + o(1))n$ for sufficiently large $n$.  On the other hand, it is unknown whether $f(n)$ is bounded by a constant. However, we do know that if $f(n)$ is bounded by a constant, then 
		the smallest such constant is at most $9$. This lecture note gives self-contained proofs of  the  above mentioned results.
	\end{abstract}
	
	\section{Introduction}

	The \emph{ product}    $G \times H$ of  graphs $G$ and $H$  has vertex set $V(G) \times V(H)$ and has  $(x, y)$ adjacent to $(x',y')$   if and only if $xx' \in E(G)$ and $yy' \in E(H)$.  Many names for this product are used in the literature, including   the \emph{categorical product}, the \emph{tensor product} and the 
	\emph{direct product}. It is the most important product in this lecture note. We just call it \emph{the product}. Edges in a graph $G$ are denoted by $xy \in E(G)$ or $x \sim y$ (in $G$).
	
	 A proper colouring $\phi$ of $G$  induces a proper colouring $\Phi$  of $G \times H$ defined as $\Phi(x,y) = \phi(x)$. So $\chi(G \times H) \le   \chi(G)$.  Symmetrically, we also have 
	 $\chi(G \times H) \le \chi(H)$. Therefore $\chi(G \times H) \le \min \{\chi(G), \chi(H)\}$.
	 In 1966, 
	 Hedetniemi conjectured in  \cite{Hedetniemi} that $\chi(G \times H) = \min \{\chi(G), \chi(H)\}$  for all graphs $G$ and $H$. This conjecture received a lot of attention in the past half century  (see \cite{ES,Klav, Sauer,Tardif,Zhu,Zhu-frac}). Some special cases are confirmed.  In particular, it is known that if 
	 $\min\{\chi(G), \chi(H)\} \le 4$, then the conjecture holds \cite{ES}. Also, a fractional version of Hedetniemi's conjecture is true \cite{Zhu-frac}. However,  Shitov recently refuted Hedetniemi's conjecture    \cite{Shitov}. He proved that for sufficiently large $n$, there are $n$-chromatic graphs $G$ and $H$ with 
	 $\chi(G \times H) < n$.  
	  
	  	The Poljak-R\"{o}dl function \cite{PR} is defined as 	
	  	$$f(n) = \min \{\chi(G \times H): \chi(G)=\chi(H)=n\}.$$
	  	Hedetniemi's conjecture is equivalent to saying $f(n)=n$ for all positive integer $n$. Shitov's result shows that  $f(n) < n$ for sufficiently large $n$. Right after Shitov put his result on arxiv, 
	  	using his result, Tardif and Zhu \cite{TZ2019} showed that the difference $n-f(n)$ can be arbitrarily large. Indeed, they proved that  $f(n) \le n- (\log n)^{1/4}$ for sufficiently large $n$. It is also shown in \cite{TZ2019} that if    Stahl's conjecture in \cite{Stahl} on the multi-chromatic number of Kneser graphs is true, then $\lim_{n \to \infty} f(n)/n \le 1/2$. He--Wigderson, using Shitov's method, proved that $f(n) \le (1- \epsilon)n$ for 
	  	$\epsilon \approx 10^{-9}$ and suffciently large $n$. 
	  	In this note we prove that the conclusion  $\lim_{n \to \infty} f(n)/n \le 1/2$ holds without assuming Stahl's conjecture. 
	  	
	  	\section{Exponential graph}
	  	
	  	One of the standard tools used in the study of Hedetniemi's conjecture is the concept of \emph{exponential graphs}.  Let $c$ be a positive integer. We denote by $[c]$ the set 
	  	$\{1,2,\ldots, c\}$.  For a graph $G$, the exponential graph 
	  	$K_c^G$ has  vertex set $$\{f: f \text{ is a mapping from $V(G) \to [c]$}\},$$
	  	 with $f \sim g$ if and only if for any edge $xy \in E(G)$, $f(x) \ne g(y)$. 
	  		In particular, $f\sim f$ is a loop in $K_c^G$ if and only if $f$ is a proper $c$-colouring of $G$.
	  		So if $\chi(G) > c$, then $K_c^G$ has no loop. 
	  		
	  		For convenience, when we study properties of $K_c^G$, vertices in $K_c^G$ will be called \emph{ maps}. The term ``vertices" is reserved for 
	  	vertices of $G$. We write a map in $K_c^G$ or a map from $G$ to $[c]$ when we refer to a vertex of $K_c^G$.  For a map $f \in K_c^G$, the image set of $f$ is $Im(f) = \{f(v): v \in V(G)\}$.
	  	Note that for $f, g \in K_c^G$, if $Im(f) \cap Im(g) = \emptyset$, then $f \sim g$. 
	  	For $i \in [c]$, we denote by $g_i$ the constant map $g_i(v)=i$ for all $v \in V(G)$. 
	  	So $Im(g_i) = \{i\}$. Thus $\{g_i: i\in [c]\}$ induces a $c$-clique in $K_c^G$ for any graph $G$ and $\chi(K_c^G) \ge c$. 
	  	
	  	For any graph $G$, the mapping $\Psi: V(G \times K_c^G) \to [c]$ defined as $\Psi(x, f) = f(x)$ is a proper colouring of $G \times K_c^G$. So $\chi(G \times K_c^G) \le c$. 
	  	If $\chi(G) > c$ and $\chi(K_c^G) > c$, then we would have a counterexample to Hedetniemi's conjecture
	  	(with $H=K_c^G$).
	  	
	  	For two graphs $G$ and $H$, a \emph{homomorphism from $G$ to $H$} is a mapping $\phi: V(G) \to V(H)$ that preserves edges, i.e., for every edge $xy$ of $G$, $\phi(x)\phi(y)$ is an edge of $H$. We say $G$ is \emph{homomorphic} to $H$, and write $G \to H$,  if there is a homomorphism from $G$ to $H$. The ``homomorphic" relation ``$\to$"  is a quasi-order. It is reflexive and transitive: if $G \to H$ and $H \to Q$ then $G \to Q$.  The composition $\psi \circ \phi$ of a homomorphism 
	  	$\phi$ from $G$ to $H$ and a homomorphism $\psi$ from $H$ to $Q$ 
	  	is a homomorphism from $G$ to $Q$. 
	  	
	  	 Note that a homomorphism from a graph $G$ to $K_c$ is equivalent to a proper $c$-colouring of $G$. Thus is $G \to H$, then $\chi(G) \le \chi(H)$. 
	  	
	  	For any graph $H$, if $\Psi: V(G \times H) \to [c]$ is a proper colouring, then the mapping
	  	sending $u \in V(H)$ to $f_u \in K_c^G$ defined as $f_u(v)=\Psi(u,v)$ is a homomorphism from $H$ to $K_c^G$. In this sense, $K_c^G$ is the largest graph $H$ in the order of homomorphism with the property 
	  	that $\chi(G \times H) \le c$.

	  	Thus if $\chi(K_c^G) =c$, then for any graph $H$ with $\chi(H) > c$, $\chi(G \times H) > c$.  So Hedetniemi's conjecture is equivalent to the following statement:
	  	
	  	\bigskip
	  	{\em If $\chi(G) >c$, then $\chi(K_c^G)=c$.}
	  	\bigskip
	
    The concept of  exponential graphs was first used in \cite{ES}, where it is shown that 
  if $\chi(G) \ge 4$, then $K_3^G$ is $3$-colourable. Hence the product of two $4$-chromatic graphs has chromatic number $4$. 
  
  The result of El-Zahar and Sauer is still the best result in the positive direction of Hedetniemi's conjecture. We do not know whether or not  the product of two $5$-chromatic graphs equals $5$. 
  On the other hand, there is a nice strengthening of this result by Tardif \cite{Tardif2} in the study of multiplicative graphs. We say a graph $Q$ is \emph{multiplicative} if for any two graphs $G,H$, $G \not\to Q$ and $H \not\to Q$ implies that $G \times H \not\to Q$.   Hedetniemi's conjecture is equiavelnt to say that $K_n$ is multiplicative for any positive integer $n$. El-Zahar and Sauer proved that $K_3$ is multiplicative. H\"{a}ggkvist, Hell, Miller and Neumann Lara \cite{HHMN} proved that odd cycles are multiplicative and Tardif \cite{Tardif2} proved that circular cliques $K_{p/q}$ for $p/q < 4$ are multiplicative, where $K_{p/q}$ has vertex set $[p]$ with $i\sim j$ if and only if $q \le |i-j| \le p-q$. (So $K_{p/1}=K_p$ and $K_{(2k+1)/k}=C_{2k+1}$ ).
  
This is a little bit astray from the main track. We come back to the proofs.

\section{Shitov's Theorem}

We denote by  $G[K_q]$ the graph  obtained from $G$ by \emph{blowing up} each vertex of $G$ into a $q$-clique. The vertices of $G[K_q]$ are denoted by $(x, i)$, where $x \in V(G)$ and $i \in [q]$.  So $(x,i)$ and $ (y,j)$ are adjacent in $G[K_q]$ if and only if 
either $x \sim y$ or $x=y$ and $i \ne j$. For a graph $G$, the 
{\em independence number}  $\alpha(G)$ of $G$ is the   size of a largest independent set in $G$.
 This section proves the following result of Shitov:

	\begin{theorem}[Shitov]  
		\label{thm-shitov}
		Assume $G$ is a graph with $|V(G)|=p$, $\alpha(G) \le \frac{p}{4.1}$ and ${\rm girth}(G) \ge 6$. Let  
		 $q \ge 2^{p-1}p^2$ and $c=4q+2$. 
		 Then  $\chi(G[K_q]) > c$ and 
		$\chi \left  ( K_c^{G[K_q]} \right ) > c$.
	\end{theorem}
	\begin{proof}
		 The above formulation of the theorem is slightly different from the formulation in \cite{Shitov}. The proof   also seems  different. But all the   claims and lemmas are either stated in \cite{Shitov} or hidden in the text in \cite{Shitov}.

		 It is a classical result of Erd\H{o}s \cite{Erdos} that there are graphs of arbtirary large girth and large  chromatic number. This result is included in most graph theory textbooks (see	  \cite{West}).  
		The probabilistic proof of this result actually shows that there are graphs $G$ of arbitrary large girth and arbtirary small independence ratio $\alpha(G)/|V(G)|$. What we need here is a graph of girth $6$ and with $\alpha(G) \le |V(G)|/4.1$,   As $G[K_q]$ has the same independence number as $G$, we conclude that $$\chi(G[K_q]) \ge \frac{ |V(G[K_q])|}{   \alpha(G[K_q])} = \frac{ |V(G)|q }{  \alpha(G)} \ge 4.1q > c.$$
		 
		 Now we shall show that $\chi(K_c^{G[K_q]}) > c$. 
		 
		 Assume to the contrary that  $\chi(K_c^{G[K_q]}) =  c$ (recall that $K_c^{G[K_q]}$ has a $c$-clique and hence has chromatic number at least $c$), and  $\Psi$ is $c$-colouring of $K_c^G$. 
		 We may assume that the constant map  $g_i$ is coloured by colour $i$. 
		 Thus for any map $\phi \in K_c^{G[K_q]}$, if $i \notin Im(\phi)$, then $\phi \sim g_i$ and hence 
		 $\Psi(\phi) \ne i$. Thus we have the following observation.
		 
		 \begin{observation}
		 	\label{ob-image}
		 	For any map $\phi \in K_c^{G[K_q]}$, $\Psi(\phi) \in Im(\phi)$.
		 \end{observation}
		 
		 \begin{definition}
		 	  A map   $\phi \in K_c^{G[K_q]}$ is called {\em simple} if $\phi$ is constant on 
		 	  each copy of $K_q$ that is a blow-up of a vertex of $G$, i.e., for any $x \in V(G), i,j \in [q]$,
		 	  $\phi(x,i)=\phi(x,j)$.
		 \end{definition}
		 For simplicity, we shall write $\phi(x)$ for $\phi(x,i)$  when $\phi$ is a simple map. 
		 
		 Note that in $K_c^{G[K_q]}$, two simple maps $\phi$ and $\psi$ are adjacent if and only if 
		 for each edge $xy$ of $G$, $\phi(x) \ne \psi(y)$, and moreover, for each vertex $x$, $\phi(x) \ne \psi(x)$. This is so, because for $i \ne j \in [q]$, $(x,i)(x,j)$ is an edge of $G[K_q]$ and 
		 $\phi(x)$ is a shorthand for $\phi(x,i)$ and $\psi(x)$ is a shorthand for $\psi(x,j)$. 
		 
		 In this sense, the subgraph of $K_c^{G[K_q]}$ induced by simple maps is isomorphic to $K_c^{G^o}$, where 
		 $G^o$ is obtained from $G$ by adding a loop to each vertex of $G$. We shall just treat $K_c^{G^o}$ as an induced subgraph of $K_c^{G[K_q]}$ and write 
		 $\phi \in K_c^{G^o}$ to mean that $\phi$ is a simple map in $K_c^{G[K_q]}$.
		 Most of our argument is about properties of the subgraph $K_c^{G^o}$ of $K_c^{G[K_q]}$.

		 The graph $K_c^{G[K_q]}$ is a huge graph. 
		 As $G$ has girth $6$ and fractional chromatic number at least $4.1$,  $p=|V(G)|$ is probably about $200$. The number in $K_c^{G[K_q]}$ is $c^{pq}$, which is roughly $(2^{200})^{2^{200}}$.
		  The subgraph $K_c^{G^o}$ has $c^p$ vertices, which is roughly $(2^{200})^{200}$.
		  So $K_c^{G^o}$ is huge,  but it is a very tiny fraction of $K_c^{G[K_q]}$.
		 
		 \begin{definition}
		 	For $v \in V(G)$ and $b \in [c]$, let $$I(v,b)=\{\phi \in K_c^{G^o}: \Psi(\phi)=b = \phi(v)\}.$$		 
		 	\end{definition}
		 
Since $\Psi(\phi) \in Im(\phi)$ for any $\phi \in K_c^{G^o}$, we conclude that 
$$V(K_c^{G^o})  = \bigcup_{v \in V(G), b \in [c]}I(v,b).$$
		 As $K_c^{G^o}$ has $c^p$ vertices, the average size of $I(v,b)$ is 
		 $$\frac{c^p}{pc} = \frac{c^{p-1}}{p}.$$
		 	
		 	\begin{definition}
		 		We say $I(v,b)$ is {\em large} if $|I(v,b)| \ge 2pc^{p-2}$.
		 	\end{definition}
		 
		 Observe that $c$ is much  larger than $p$. The powers of $c$ is the dominating factor. So $2pc^{p-2}$ is much smaller than the average size of $I(v,b)$. Thus intuitively,  ``most" of the $I(v,b)$'s should be large. So the next lemma is not a surprise.

	\begin{lemma}
		\label{large}
		There exists a vertex $v$ of $G$ such that $$|\{b \in [c]: I(v,b) \text{ is large } \}| > c/2.$$
	\end{lemma}
	\begin{proof}
		For each vertex $v$ of $G$, let $S(v) = \{b: I(v,b) \text{ is small}\}$. Assume to the contrary that for each $v$, $|S(v)| \ge c/2$. Let 
		$${\cal L} = \{\phi \in K_c^{G^o}: \forall v \in V(G), \phi(v) \in S(v)\}.$$
		Then $$|{\cal L}| = \prod_{v \in V(G)}|S(v)| \ge \left( \frac c2 \right)^p.$$
		For any $\phi \in {\cal L}$, if $\phi \in I(v,b)$, then $I(v,b)$ is small. 
		Thus $${\cal L} \subset \bigcup_{v \in V(G), b \in [c], I(v,b) \text{ is small}}I(v,b).$$
		Therefore $|{\cal L}| < p\times c \times 2pc^{p-2} = 2p^2c^{p-1}$. But then 
		$$\left( \frac c2 \right)^p <  2p^2c^{p-1}$$
	which	implies that $c < 2^{p+1}p^2$. But by our choice of $c$, we have $c = 4q+2 > 4q \ge 2^{p+1}p^2$, a contradiction.
	\end{proof}
	
Let $v$ be a vertex of $G$ for which $|\{b \in [c]: I(v,b) \text{ is large } \}| > c/2.$
For $t \in \{2q+1, 2q+2, \ldots, 4q+2\}$, let $\mu_t \in K_c^{G[K_q]}$ be defined as
\[
\mu_t(x, i) = \begin{cases} i, & \text{ if $d_G(x,v)=0,2$,} \cr
q+i,  & \text{ if $d_G(x,v)=1$,} \cr
t, & \text{ if $d_G(x,v)\ge 3$}.
\end{cases}
\]

Observe that $\mu_t$ are not simple maps. These will be the only non-simple maps used in the proof.

\begin{claim}
	\label{clm-clique}
	The set of maps $\{\mu_t: t=2q+1, 2q+2, \ldots, 4q+2\}$ induces a clique in $K_c^{G[K_q]}$.
\end{claim}
\begin{proof}
	Assume to the contrary that for some $t \ne t'$, $\mu_t \not\sim \mu_{t'}$. Then there is an edge $(x,i)(y,j)$ of $G[K_q]$ such that $\mu_t(x,i) = \mu_{t'}(y,j)$. 
	Let $\alpha = \mu_t(x,i) = \mu_{t'}(y,j)$. 
	
	Then $\alpha \in Im(\mu_t) \cap Im(\mu_{t'})\subseteq  \{i, q+i, t\} \cap \{j, q+j, t'\}$.
	As $t \ne t'$, we conclude that $i=j$ and $\alpha=i$ or $q+i$. Since 
	$(x,i), (y,i)$ are distinct adjacent vertices, we conclude that
	 $x \ne y$ and $xy \in E(G)$. If $\alpha=i$, then 
	$d_G(x,v), d_G(y,v) \in \{0,2\}$ implies that $G$ has a $3$-cycle or a $5$-cycle, contrary to the assumption that $G$ has girth $6$. If $\alpha = q+i$, then $d_G(v,x)=d_G(v,y)=1$, and $G$ has a 3-cycle, again a contradiction.	This completes the proof of Claim \ref{clm-clique}.
\end{proof}

So maps $\{\mu_t: t=2q+1, 2q+2, \ldots, 4q+2\}$  are coloured by distinct colours, and hence there exists $t$ such that $\Psi(\mu_t) \not\in \{1,2,\ldots, 2q\}$. As $\Psi(\mu_t) \in Im(\mu_t) = \{1,2,\ldots, q, t\}$,
we  have $\Psi(\mu_t) = t$. 

Since $|\{b \in [c]: I(v,b) \text{ is large } \}| > c/2 = 2q+1$, there is a colour $b \in [c] - \{1,2,\ldots, 2q, t\}$ such that $I(v,b)$ is large. Let $\theta \in K_c^{G^o}$ be defined as follows:

\[
\theta(x) = \begin{cases} b, & \text{ if $d_G(x,v)\ge 2$,} \cr
t, & \text{ if $d_G(x,v)\le 1$}.
\end{cases}
\]

\begin{claim}
	\label{clm-ad}
$\theta \sim \mu_t$.
\end{claim}
\begin{proof}
	Assume to the contrary that $\theta \not\sim \mu_t$. Then there is an edge 
	$(x,i)(y,j) \in E(G[K_q])$ such that $\theta(x) = \theta(x,i) = \mu_t(y,j)$. (Note that $\theta(x,i)=\theta(x)$ as $\theta$ is a simple map). 
	As $Im(\theta) \cap Im(\mu_t) = \{t\}$, we conclude that $\theta(x) =  \mu_t(y,j) = t$.
	But then $d_G(x,v) \le 1$ and $d_G(y,v) \ge 3$, and hence $x \ne y$ and $xy \notin E(G)$, contrary to the assumption that 
	$(x,i)(y,j) \in E(G[K_q])$.
	\end{proof}
	
	Thus $\Psi(\theta) \ne \Psi(\mu_t) = t$. As $\Psi(\theta) \in Im(\theta)$, we conclude that $\Psi(\theta) = b$. 
	
	\begin{claim}
		\label{clm-in}
		For any $\phi \in I(v,b)$, there exists a vertex $x \ne v$ such that $\phi(x) \in \{b,t\}$.
	\end{claim}
	\begin{proof}
		 Assume $\phi \in I(v,b)$. By definition $\Psi(\phi) = b = \phi(v)$. 
		 So $\Psi(\phi) = \Psi(\theta)$. Hence $\phi \not\sim \theta$. So there is an edge 
		 $xy \in E(G^o)$ such that $\phi(x) = \theta(y)$. If $x = v$, then $\theta(y) = \phi(v)=b$.
		 By definition of $\theta$, we have $d_G(y,v) \ge 2$. Hence $xy$ cannot be an edge in $G^o$, a contradiction. 
		 So $x \ne v$. As $\phi(x) = \theta(y) \in \{b,t\}$, this completes the proof of the claim.	 
	\end{proof}
	
	For each $x \ne v$, let $$J_x = \{\phi(x) \in I(v,b): \phi(x) \in \{b,t\}\}.$$
	As $\phi(v)=b$ for $\phi \in I(v,b)$, we conclude that $|J_x| \le 2c^{n-2}$. 
	By Claim \ref{clm-in}, $I(v,b) = \cup_{x \in V(G)-\{v\}}J_x$. So 
	$|I(v,b)| \le 2(n-1)c^{n-2}$, contrary to the assumption that $I(v,b)$ is large.
	This completes the proof of Theorem \ref{thm-shitov}.
		\end{proof}
		
		\begin{remark}
			The key part of the proof of Theorem \ref{thm-shitov} is to show that $K_c^{G[K_q]}$  is not $c$-colourable. For each vertex $v$ of $G$, for $t \in \{2q+1, 2q+2, \ldots, 4q+2\}$, let 
			\[
			\mu_{v,t}(x, i) = \begin{cases} i, & \text{ if $d_G(x,v)=0,2$,} \cr
			q+i,  & \text{ if $d_G(x,v)=1$,} \cr
			t, & \text{ if $d_G(x,v)\ge 3$};
			\end{cases}
			\]
			Let $H$ be the subgraph of $K_c^{G[K_q]}$ induced by  $$V(K_c^{G^o}) \cup \{\mu_{v,t}: v \in V(G),  t \in \{2q+1, 2q+2, \ldots, 4q+2\}  \}.$$
			What we have proved is that the subgraph $H$ of $K_c^{G[K_q]}$ is not $c$-colourable. 
			Note that $H$ is a very tiny fraction of $K_c^{G[K_q]}$, although $H$ by itself is a huge graph. Possibly, the chromatic number of   $K_c^{G[K_q]}$ is much larger than $c$.
			\end{remark}
	
	\section{The Poljak-R\"{o}dl function}
	
	The Poljak-R\"{o}dl function is defined in \cite{PR}:
	$$f(n) = \min\{\chi(G \times H): \chi(G), \chi(H) \ge n\}.$$
	Hedetniemi's conjecture is equivalent to say that $f(n) =n$ for all psoitive integer $n$. Shitov's Theorem says that for sufficiently large $n$, $f(n) \le n-1$. Using Shitov's result,
	Tardif and Zhu \cite{TZ2019} proved that $f(n) \le n - \left(\log n\right)^{1/4}$. Tardif and Zhu asked in \cite{TZ2019} if there is a positive constant $\epsilon $ such that $f(n) \le (1-\epsilon)n$ for sufficiently large $n$. This question was answered in affirmative by He and Wigderson \cite{HW2019} with $\epsilon \approx 10^{-9}$. On the other hand, in \cite{TZ2019}, Tardif and Zhu proved that if a special case of a conjecture of Stahl \cite{Stahl} conserning the multi-chromatic number of Kneser graph is true, then we have $\lim_{n \to \infty} \frac{f(n)}{n} \le \frac 12$.  
	
	In this note, we that the conclusion  $\lim_{n \to \infty} \frac{f(n)}{n} \le \frac 12$ holds  without assuming Stahl's conjecture.

\begin{theorem}
	\label{thm-main}
	For $d \ge 1$, let $G$ be a graph of girth $6$ and with $\chi_f(G) \ge 8.1d$. Let $p=|V(G)|$, $q \ge  2^{p-1}p^2$ and $c = 4q+2$.
	Then $\chi(G[K_q]) \ge 2dc - 2c+2$ and 
	$\chi( K_{dc}^{G[K_q]}) \ge 2dc-2c+2$.
		Consequently, $f(2dc - 2c+2) \le dc$.
\end{theorem}
\begin{proof}
	Similarly as in the proof of Theorem \ref{thm-shitov},  $\chi(G[K_q]) \ge \chi_f(G) q \ge 8.1d q \ge 2dc > 2dc-2c+2$. Now we show that $\chi( K_{dc}^{G[K_q]}) \ge 2dc-2c+2$.
	
	Assume $\Psi$ is a $(dc+t)$-colouring of $K_{dc}^{G[K_q]}$ with colour set $[dc+t]$. We shall show that 
	$dc+t \ge  2dc-2c+2$, i.e., $t \ge dc-2c+2$. Let $S=[dc+t]- [dc]$. The colours in $[dc]$ are called {\em primary colours} and colours in $S$ are called {\em secondary colours}. So we have $t=|S|$ secondary colours.
	
	Similarly as before, we may assume that $\Psi(g_i)=i$ for $i \in [dc]$.
	Then for any map $\phi \in K_{dc}^{G[K_q]}$, if $i \notin Im(\phi)$, then $\phi \sim g_i$ and $\Psi(\phi) \ne i$. Thus for any $\phi \in K_{dc}^{G[K_q]}$, $\Psi(\phi) \in Im(\phi) \cup S$. 
	
	For positive integers $m \ge 2k$, $K(m, k)$ is the Kneser graph whose vertices are $k$-susbets of $[m]$, and for 
	two $k$-subsets $A, B$ of $[m]$, 
  $A \sim B$ if $A \cap B = \emptyset$. 
  It was proved by Lov\'asz in \cite{Lovasz} that 
  $\chi(K(m,k))= m-2k+2$.
	
	For a $c$-subset $A$ of $[cd]$, let $H_A$ be the subgraph of $K_{cd}^{G[K_q]}$ induced by 
	$$\{\phi \in V(K_{cd}^{G[K_q]}): Im(\phi) \subseteq A\}.$$
	Then $H_A$ is isomorphic to $K_c^{G[K_q]}$.
	By Theorem \ref{thm-shitov}, $|\Psi(H_A)| \ge c+1$.  As $|Im(\phi)|=c$, $\Phi(H_A)$ contains at least one secondary colour. Let $\tau(A)$ be an arbitrary secondary colour contained in $\Psi(H_A)$. 
	
	If $A, B$ are $c$-subsets of $[dc]$ and $A \cap B = \emptyset$, then every vertex in $H_A$ is adjacent to every vertex in $H_B$. Hence $\Psi(H_A) \cap \Psi(H_B) = \emptyset$. In particular, $\tau(A) \ne \tau(B)$. Thus $\tau$ is a proper colouring of the Kneser graph $K(dc, c)$. As $\chi(K(dc, c)) = dc-2c+2$, we conclude that $t=|S| \ge dc-2c+2$. This completes the proof of Theorem \ref{thm-main}. 	
\end{proof}

For a  positive integer $d$, let $p=p(d)$ be the minimum number of vertices of a graph $G$ with girth $6$ and $\chi_f(G) \ge 8.1d$. 
It follows from Theorem \ref{thm-main} that for any integer $q \ge p^22^{p-1}$, 
$f(2(d-1) (4q+2) +2) \le (4q+2) d$.
As $f(n)$ is non-decreasing, for integers $n$ in the interval   $[2(d-1) (4q+2)  +2,   2(d-1) (4q+6) +2]$,
we have $f(n) \le  (4q+6) d$.

Hence for all integers $n \ge  2 (4q+2) (d-1)+2$,  $$\frac{f(n)}{n} \le  \frac{(4q+6)d}{2(4q+2)(d-1)+3} = \frac 12 +  \frac {4q+ 4d+1}{ 2(d-1) (4q+2)  +2}.$$
Note that if $d \to \infty$, then $p=p(d)$ goes to infinity, and hence  $q \ge p^32^p $ goes to infinity.  Hence 
$$\lim_{n \to \infty} \frac{f(n)}{n} \le \frac 12.$$
 
Theorem \ref{thm-main} improves the result of He and Wigderson \cite{HW2019}. However, He and Wigderson uses a modification of Shitov's method, which might be of independent interest.

			In the proof of Theorem \ref{thm-main}, we actually showed that a 
			tiny subgraph of $K_{dc}^{G[K_q]}$ has chromatic number close to 
			$2dc$.  It is not clear if the remaining part of the graph $K_{dc}^{G[K_q]}$
		 can be used to show that this graph actually have a much larger chromatic number. We observe that   if one can show that the chromatic number of $K_{dc}^{G[K_q]}$ is more than $kdc$ for some positive integer $k$, then Stahl's conjecture implies that $\lim_{n \to \infty} \frac{f(n)}{n} \le \frac{1}{k+1}  $. 
	 
		 \section{Lower bound for $f(n)$} 
		 
		 The breakthrough result of Shitov leads to improvement of the upper bound for the function $f(n)$. On the other hand, the only known lower bound for $f(n)$ is that $f(n) \ge 4$ for $n \ge 4$. We do not know if $f(n)$ is bounded by a constant or not. 
		 What we do know is that if $f(n)$ is bounded by a constant, then the smallest such constant is at most $9$. 
		 
		 To prove this result, we need to consider the product of digraphs. For a digraph $D$, we use $A(D)$ to denote the set of arcs of $D$. An arc in $D$ is either denoted by an ordered pair $(x,y)$, or by an arrow
		 $x \to y$. Digraphs are allowed to have digons.
		 
		  Assume $D_1, D_2$ are digraphs.
		 The product $D_1 \times D_2$ has vertex set $V(D_1) \times V(D_2)$, with
		 $(x,y) \to (x',y')$ be an arc if and only if $(x,x')$ is an arc in $D_1$ and $(y,y')$ is an arc in $D_2$. The chromatic number of a digraph $D$ is defined to be $\chi(\underline{D})$, where $\underline{D}$ is the underline graph of $D$, i.e., obtained from $D$ by replacing each arc $(x,y)$ with an edge $xy$. Given a digraph $D$, let $D^{-1}$ be the digraph obtained from $D$ by reversing the direction of all its arcs. It is easy to see that 
		 for any digraphs $D_1, D_2$, $$\underline{D_1} \times \underline{D_1} = (\underline{D_1 \times D_2})
\cup (\underline{D_1 \times D_2^{-1}}).$$	
Hence $$\chi (\underline{D_1} \times \underline{D_1}) \le \chi (D_1 \times D_2) \times \chi (D_1 \times D_2^{-1}).$$
Let 
\begin{eqnarray*}
g(n) &=& \min\{\chi(D_1 \times D_2): \chi(D_1), \chi(D_2) \ge n\},\\
h(n) &=& \min\{\max\{\chi(D_1 \times D_2), \chi(D_1 \times D_2^{-1})\}: \chi(D_1), \chi(D_2) \ge n\}.
\end{eqnarray*}	
It follows from the discussion above that 
$$g(n) \le h(n) \le  f(n) \le h(n)^2.$$ 

The following result was proved by Poljak and R\"{o}dl in \cite{PR}.

\begin{theorem}
	\label{thm-lower}
	If $g(n)$ (respectively $h(n)$) is bounded by a constant, then the smallest such constant is at most $4$. 
	Consequently, if $f(n)$ is bounded by a constant,  then the smallest  such constant is at most $16$.
\end{theorem}
\begin{proof}
	For a graph $D$, let $\partial(D)$ be the digraph with vertex set $A(D)$, with $(x,y) \to (x',y')$ be an arc of $\partial(D)$ if and only if $y=x'$. In particular, if $(x,y), (y,x)$   is a digon in $D$, then $(x,y) \to (y,x)$ and $(y,x) \to (x,y)$ is a digon in $\partial(D)$.
	
	\begin{lemma}
		\label{lem-rel}
		For any digraph $D$, $$\min\{k: 2^k \ge \chi(D)\} \le \chi(\partial(D)) \le \min\{k: {k \choose \lceil k/2 \rceil} \ge \chi(D)\}.$$
	\end{lemma}
	\begin{proof}
		If $\phi: V(\partial(D)) \to [k]$ is a proper colouring of $\partial(D)$, then for each vertex $v$ of $D$, let $\psi(v)= \{\phi(e): e \in A^+(v)\}$, where $A^+(v)$ is the set of out-arcs at $v$. Then $\psi$ is a proper colouring of $D$ (with subsets of $[k]$ as colours). Indeed, if $e=(x, y)$ is an arcs of $D$, then $\phi(e)  \in \psi(x) - \psi(y)$. So $\psi(x) \ne \psi(y)$. The number of colours used by $\psi$ is at most the number of subsest of $[k]$, which is   $2^k$. 
		
		If $\psi: V(D) \to {k \choose \lceil k/2 \rceil }$ is a proper colouring of $D$ (where the colours are $ \lceil k/2 \rceil$-subsets of $[k]$), then for any arc 
		$e=(x,y)$ of $D$, let $\phi(e)$ be any integer in $\psi(y)-\psi(x)$ (as $\psi(y) \ne \psi(x)$, such an integer exists). Then if $(x,y) \to (y,z)$ is an arc in $\partial(D)$, then 
		$\phi(x,y) \in \psi(y)$ and $\phi(y,z) \notin \psi(y)$. Hence $\phi(x,y) \ne \phi(y,z)$. I.e., $\phi$ is a proper colouring of $\partial(D)$. This completes the proof of Lemma \ref{lem-rel}. 	
	\end{proof}
	
	It follows easily from the definition that 
	\begin{eqnarray*}
		\partial(D_1 \times D_2) &=& \partial(D_1) \times \partial(D_2),\\
		\partial(D^{-1})&=&  (\partial(D))^{-1}.
	\end{eqnarray*}	
	
	Suppose $g(n)$ is bounded and $C$ is the smallest upper bound. As $g(n)$ is non-decreasing, there is an integer $n_0$ such that $g(n)=C$ for all $n \ge n_0$. Let $n_1=2^{n_0}$, and let $D_1, D_2$ be digraphs
	with $\chi(D_1), \chi(D_2) \ge n_1$ and $\chi(D_1 \times D_2) = C$. It follows from Lemma \ref{lem-rel} that 
	$\chi(\partial(D_1)),\chi(\partial(D_2)) \ge n_0$ and hence 
	$\chi(\partial(D_1) \times \partial(D_2)) = \chi(\partial(D_1 \times D_2))  \ge C$. 
	By Lemma \ref{lem-rel} again, we have 
	$$C \ge \chi(D_1 \times D_2) > {C-1 \choose \lceil (C-1)/2 \rceil}.$$
	This implies that $C \le 4$. 
	
	The same argument shows that if $h(n)$ is bounded by a constant, then the smallest such constant is at most $4$. Since $h(n) \le f(n) \le h(n)^2$, if $f(n)$ is bounded by a constant, then the smallest such a constant is at most $16$.
		 \end{proof}
		 
		 Next we show that if $g(n)$ (respectively, $h(n)$)  is bounded by a constant, then the smallest such constant cannot be $4$. Assume to the contrary that the smallest constant bound for $g(n)$ is $4$. 
		 Let $n_0$ be the integer given above, and let $n_1=2^{n_0}, n_2=2^{n_1}$. Then 
		 $g(n_2)=g(n_1)=g(n_0)=4$. Let $D_1, D_2$ be two digraphs with $\chi(D_1),\chi(D_2) \ge n_2$
		 and $\chi(D_1 \times D_2) = 4$. The same argument as above shows that 
		 $$\chi(\partial(\partial(D_1 \times D_2)))=4.$$
	 
		 However, we shall show that if $\chi(D) \le 4$, then $\chi(\partial(\partial(D))) \le 3$. 
		 Let $\vec{K}_4$ be the complete digraph with vertex set $\{1,2,3,4\}$, with $(i,j)$ be an arc   for any distinct $i,j \in \{1,2,3,4\}$.  If $\chi(D) =4$, then $D$ admits a homomorphism to $\vec{K}_4$.
		 Hence $\partial(\partial(D))$ admits a homomorphism to $\partial(\partial(\vec{K}_4))$. 
		 So it suffices to show that $\partial(\partial(\vec{K}_4)) \le 3$. In 1990, I was a Ph.D. 
		 student at The University of Calgary. After reading the paper by Poljak and R\"{o}dl \cite{PR}, 
		 I found a $3$-colouring of $\partial(\partial(\vec{K}_4))$ by brute force.
		 I was happy to tell this to my supervisor Professor Norbert Sauer, who then told the result to Duffus. Then I learned from Duffus the following elegant $3$-colouring of $\partial(\partial(\vec{K}_4))$, given earlier by Schelp and was not published.
		 
		 Each vertex of $\partial(\partial(\vec{K}_4))$ is a sequence $ijk$ with $i,j,k \in [k]$, $i \ne j, j \ne k$ (but $i$ may equal to $k$). Let 
		 
		 \[
		 c(ijk) = \begin{cases} j, & \text{ if $j \ne 4$,} \cr
		 s, & \text{ if $j=4$ and $s \in \{1,2,3\}-\{i,k\}$}
		 \end{cases}
		 \]
		 Then it is easy to verify that $c$ is a proper $3$-colouring of $\partial(\partial(\vec{K}_4))$.
		 This completes the proof that $g(n)$ is either bounded by $3$ or goes to infinity. Similarly, 
		 $h(n)$ is either bounded by $3$ or goes to infinity, and consequently, $f(n)$ is either bounded by $9$ or goes to infinity. 
		 
		 Later I learned from Hell that Poljak 
		 also obtained this strengthening independently and that was published in 1992 \cite{Poljak}.
		 
		 Tardif and Wehlau \cite{TW2004} proved that $f(n)$ is bounded if and only if $g(n)$ is bounded.
		 
		 The fractional version of Hedetniemi's conjecture was proved in  \cite{Zhu-frac}:
		For any two graphs $G$ and $H$, $\chi_f(G \times H) = \min \{\chi_f(G), \chi_f(H)\}$.
		Thus if $f(n)$ is bounded by $9$, and  $G$ and $H$ are $n$-chromatic graphs with 
		$\chi(G \times H) \le 9$, then at least one of $G$ and $H$  has   fractional chromatic number at most $9$. 
		
		In \cite{Zhu-frac}, I defined the following Poljak-R\"{o}dl type function:
		$$\psi(n) = \min \{\chi(G \times H): \chi_f(G), \chi(H) \ge n\}.$$
		I proposed a weaker version of Hedetniemi's conjecture, which is equivalent says that 
		$\psi(n)=n$ for all positive integer $n$. However, Shitov's proof actually refutes this weaker version of Hedetniemi's conjecture, as the graph $G$ used in the proof of Theorem \ref{thm-shitov}  have large fractional chromatic number. The proof of Theorem \ref{thm-main}   shows that 
		$$\lim_{n \to \infty} \frac{\psi(n)}{n} \le \frac 12.$$
		On the other hand, it follows from the definition that 
		$f(n) \le \psi(n)$. A natural question is the following: 
		
		\begin{question}
			 Is $\psi(n)$ bounded by a constant?
			 If $\psi(n)$ is bounded by a constant, what could be  the smallest such constant?
		\end{question}

\end{document}